\documentclass[a4paper]{amsart}

\pdfoutput=1

\usepackage[T1]{fontenc}
\usepackage[utf8]{inputenc}
\usepackage{microtype}
\usepackage{url}
\usepackage{hyperref}
\usepackage{tikz}
\usepackage{tikz-cd} 
\usepackage{array}
\usepackage[english]{babel}
\usepackage{amsthm}
\usepackage{amsmath}
\usepackage{calc}
\usepackage{ dsfont }
\usepackage{caption}
\usepackage{amsfonts}
\usepackage{epigraph}
\usepackage{amssymb}
\usepackage{blindtext}
\usepackage{subcaption}
\usepackage{caption}
\usepackage{tabu}

\usepackage{multicol, latexsym}
\usepackage{blindtext}
\usepackage{mdframed}
\usepackage{ bbold }
\usepackage{afterpage}
\usetikzlibrary{positioning,shadows,arrows}
\newcommand{\tree}{\mathcal{T}}
\newcommand{\leaves}{\mathcal{L}}
\newcommand{\nodes}{\mathcal{N}}
\newcommand{\vertices}{\mathcal{V}}
\newcommand{\rootvertex}{\mathbb{r}}
\newcommand{\network}{\overline{\tree}}
\newcommand{\tttree}{T\tree}
\newcommand{\ttnetwork}{\overline{T\tree}}
\newcommand{\hierarchicaltree}{H\tree}
\newcommand{\hierarchicalnetwork}{\overline{H\tree}}
\newlength{\helpspace}
\newcommand{\descendants}[2][]{%
  \setlength{\helpspace}{\widthof{$#1$}}%
  \text{$\downarrow\negthickspace^{\negthickspace^{#1}}%
    \hspace*{.5mm-2\helpspace/3}#2$}}
\newcommand{\antidescendants}[2][]{%
  \setlength{\helpspace}{\widthof{$#1$}}%
  \text{$\uparrow\negthickspace_{\negthickspace_{#1}}%
    \hspace*{.5mm-\helpspace/3}#2$}}

\newcommand{\descendantof}[1][\tree]{\le_#1}
\newcommand{\ancestorof}[1][\tree]{\ge_#1}
\DeclareMathOperator{\lca}{\bigvee} 

\DeclareMathOperator{\tns}{TNS}
\DeclareMathOperator{\tne}{\mathcal{E}}
\DeclareMathOperator{\doad}{DOAD}
\newcommand{\maxima}{m} 
\newcommand{\contract}{\mathbin{\llcorner}}

\usepackage{todonotes}
\definecolor{cSofia}{rgb}{0.1,0.45,0.03}
\definecolor{cChristian}{rgb}{0.6,0.5,0.0}

\usepackage{tikz}
\usetikzlibrary{trees}
\usepackage{float}
\newcommand{\R}{\mathds{R}}

\newcommand{\Z}{\mathds{Z}}

\newcommand{\N}{\mathds{N}}

\usepackage[inner=3cm,outer=3cm,top=3.3cm,bottom=3cm]{geometry}

\usepackage{amssymb, amsfonts}
\usepackage{mathtools}
\usepackage{comment}
\usepackage{paralist}

\usepackage{xcolor}

\usepackage[english]{babel} 
\usepackage[autostyle]{csquotes} 
\usepackage{tikz, tikz-cd}

\theoremstyle{plain}
\newtheorem{theorem}{Theorem}[section]
\newtheorem{prop}[theorem]{Proposition}
\newtheorem{corollary}[theorem]{Corollary}
\newtheorem{lemma}[theorem]{Lemma}
\newtheorem{question}[theorem]{Question}

\theoremstyle{definition}
\newtheorem{definition}[theorem]{Definition}
\newtheorem{example}[theorem]{Example}
\newtheorem{remark}[theorem]{Remark}

\newcommand{\rleft}{\mathopen{}\mathclose\bgroup\left}
\newcommand{\rright}{\aftergroup\egroup\right}

\newcommand{\kk}{\Bbbk} 

\newcommand{\ambientspaces}{\mathbb{V}}
\newcommand{\one}{\mathbb{1}}
\newcommand{\st}{\,|\,}

\title{Containments of Tensor Network Varieties}

\author{Sofía Garzón Mora}
\address{Sofía Garzón Mora, Mathematik, Freie Universität Berlin, 14195 Berlin, Germany.}
\email{sofia.garzon.mora@fu-berlin.de}
\thanks{}

\author{Christian Haase}
\address{Christian Haase, Mathematik, Freie Universität Berlin, 14195 Berlin, Germany.}
\curraddr{}
\email{haase@math.fu-berlin.de}
\thanks{}

\subjclass[2021]{Primary: ; Secondary: }
\keywords{Perfect Binary Tree, Train Track Binary Tree, Tensor format, Hackbusch Conjecture, Tensor Network States Variety}

\begin{document}
\selectlanguage{english}

\begin{abstract}
Building upon the work of Buczyńska et al.~\cite{BBM15}, we study here tensor formats and their corresponding encoding of tensors via two-fold tensor products determined by the combinatorics of a binary tree. 
The set of all tensors representable by a given network forms the corresponding tensor network variety.
A very basic question asks whether every tensor representable by one network is representable by another network, namely, when one tensor network variety is contained in another. Specific instances of this question became known as the Hackbusch Conjecture. Here, we propose a general framework for this question and take first steps, theoretical as well as experimental, towards a better understanding.
In particular, given any two binary trees on $n$ leaves, we define (and prove existence of) a new measure, the \emph{containment exponent}, which gauges how much one has to boost the parameters of one network for the containment to hold. We present an algorithm for bounding these containment exponents of tensor network varieties and report on an exhaustive search among trees on up to $n=8$ leaves.
\end{abstract}

\maketitle{}

\section{Introduction}
\label{sec:intro}

Encoding data \textit{efficiently} using tensors is common practice in many sciences. The approach we will study here makes use of \textit{tensor formats}. A tensor format encodes an $n$-way tensor as a sequence of linear subspaces via two-fold tensor products in a way determined by a binary tree with $n$ leaves. The main idea being that the number of parameters needed to determine the tensor is much smaller than the dimension of the ambient space -- the tensor product of $n$ vector spaces.

Our choice of binary tree, of a family of vector spaces for the leaves and a dimension function for the subspaces on the vertices of the tree determines a tensor network. This, together with a nesting property of these subspaces of the nodes determines the variety of tensor network states, meaning the locus of all possible tensors which can be encoded in this manner. Hence, the natural question which arises is whether one tensor network variety is contained in another. We will make sense of this question and give precise definitions in the next section.

The containments of tensor network varieties have been studied for some particular plane trees, namely the so-called hierarchical tree and the train track tree in \cite{BBM15}, where a conjecture initially posed by Hackbusch~\cite{Hac12} is resolved. Here, our aim is to generalize this containment question to more generally shaped trees, disregarding a choice of planar embedding. For this, we study the poset structure of the trees and find a combinatorial bound for the containment of their respective tensor network varieties.

In particular, given any two binary trees on $n$ leaves, we define (and prove existence of) a new measure, the \emph{containment exponent}, which gauges how much one has to boost the parameters of one network for the containment to hold. Namely, when we scale the dimension function of the first network, how much do we have to scale the dimension function of the second for the containment to hold?
We have implemented an algorithm, which we use to bound the containment exponents among all tensor network varieties that arise from full binary rooted trees with $n \le 8$ leaves. Since the bound depends on the labeling of the leaves, in order to cover all possible containment exponents, we take into account possible permutations of the labels of the $n$ leaves. The results for $n=4,5,6,7,8$ leaves are collected in \cite{sofiaGitHub}.

Additionally, we expand the previous results from the literature by introducing a combinatorial bound for the containment of varieties arising from arbitrarily shaped plane binary trees into one another. We present examples showing that the obtained bounds are not necessarily sharp and end this paper with some open questions.

\section*{Acknowledgments}
SGM was supported by the Deutsche Forschungsgemeinschaft
(DFG) under the Graduiertenkolleg ``Facets of Complexity'' (GRK 2434)
and under Germany's Excellence Strategy – The Berlin Mathematics
Research Center MATH+ (EXC-2046/1, project ID: 390685689).
CH was supported, in part, by the Polish-German grant ``ATAG --
Algebraic Torus Actions: Geometry and Combinatorics'' [project number
380241778] of the Deutsche Forschungsgemeinschaft (DFG), and 
by the SPP 2458 ``Combinatorial Synergies'', funded by the DFG.

\section{Tensor Formats for Binary Trees}
\label{sec:tensorform}
We begin by introducing basic notions regarding binary trees and the
associated varieties of network states. Then, we review results from
the literature, mainly~\cite{BBM15}, which will be crucial in our
argument.

\subsection{Trees} \label{sec:trees}
All trees in this article will be full binary, rooted trees.

\begin{definition}
A \textit{full binary tree} $\tree$ is a rooted tree with root
$\rootvertex(\tree)$ such that each vertex has either
exactly two children, in which case it is referred to as a
\textit{node}, or it has no descendants at all, so that it is called a
\textit{leaf}.

A plane tree is a tree together with a choice of which child of any
given node is left and which one is right.
\end{definition}

We denote the set of nodes (including $\rootvertex(\tree)$) by
$\nodes(\tree)$, the set of leaves by $\leaves(\tree)$.
We omit the tree from the notation if the tree which we are talking about is clear from the context.
We will also say that $\tree$ is a tree on $\leaves$.
The set of vertices is the union $\vertices=\nodes \cup \leaves$.
We have $|\nodes|=|\leaves|-1$. 

The descendant relation gives a poset structure to $\vertices$. We
write $v \descendantof w$ if $v$ is a descendant of $w$.
Following \cite{BBM15}, we will denote by $\descendants[\tree]{v} \subseteq
\leaves$ the subset of descendant leaves of $v$, and by
$\antidescendants[\tree]{v}$ the complement $\leaves \setminus
\descendants[\tree]{v}$. If it is unambiguous which tree we are
working with, we will omit it from the notation and ony write
$\antidescendants{v}$ respectively $\descendants{v}$.

Two particular classes of trees which have gotten special attention in
applications and in the literature are the hierarchical and the
train track trees:

The \textit{hierarchical tree} $\hierarchicaltree_k$ is a perfect binary
tree, i.e. a full binary tree in which every leaf has exactly the same
number $k$ of ancestors. We say the tree $\hierarchicaltree_k$ has depth
$k$, so that $|\vertices|=2^{k+1}-1$ and $|\leaves|=2^k$. 
Any two plane trees with this underlying tree are isomorphic.

\begin{figure}[H]
\centering
\begin{tikzpicture}[level distance=1.5cm,
  level 1/.style={sibling distance=3cm},
  level 2/.style={sibling distance=1.5cm}]
  \node {root}
    child {node {$v_1$}
      child {node {Leaf 1}}
      child {node {Leaf 2}}
    }
    child {node {$v_2$}
    child {node {Leaf 3}}
      child {node {Leaf 4}}
    };
\end{tikzpicture}
\caption{Hierarchical tree $\hierarchicaltree_2$ of depth 2, with 4 leaves.}
\label{fig:HT}
\end{figure}
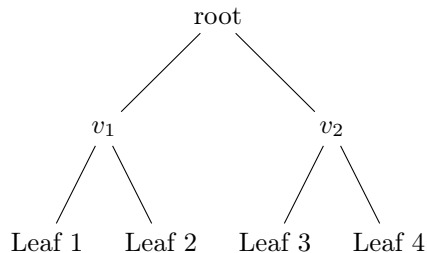

The \textit{train track tree} $\tttree_n$ with $n$ leaves is
characterized by the fact that each node has at least one leaf as a
child.
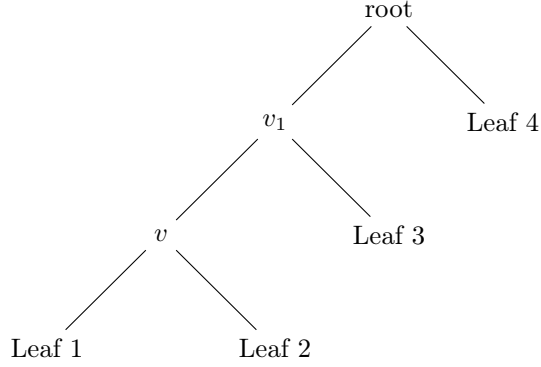
\begin{figure}[H]
\centering
\begin{tikzpicture}[level distance=1.5cm,
  level 1/.style={sibling distance=3cm},
  level 2/.style={sibling distance=3cm}
  level 3/.style={sibling distance=3cm}]
  \node {root}
    child {node {$v_1$}
      child {node {$v$}
        child {node {Leaf 1}}
       child {node {Leaf 2}}
      }
      child {node {Leaf 3}}
    }
    child {node {Leaf 4}
    };
\end{tikzpicture}
\caption{Train track tree $\tttree_4$ with 4 leaves.}
\label{fig:TT}
\end{figure}
If we refer to $\tttree_n$ as a plane tree, the left child of each node
has at least as many descendants as the right child. (There is only
one draw, and the two choices yield isomorphic plane trees.)

\subsection{Varieties of Tensor Network States} \label{sec:varieties}

We start with a finite set $\leaves$ and a family of
vector spaces\footnote{The ground field is irrelevant for
  our considerations. In applications, it will be $\R$.
}
$\ambientspaces = (V_\ell)_{\ell \in \leaves}$.
Our ambient space will then be $\bigotimes_{\ell \in \leaves}
V_\ell$.

Given a tree $\tree$ on the leaf set $\leaves$, we can successively
build a tensor in this ambient space according to $\tree$ as follows.

\begin{definition}
  A \textit{tensor network} $\network = (\tree,\ambientspaces,f)$ is a
  tree $\tree$ together with a family $\ambientspaces$ of vector spaces
  as above, and together with a
  dimension vector $f \in \N^\vertices$. We will write
  $f=\one$ for the vector with coordinates $f_v=1$ for all $v \in \vertices$.
  Later on, we will also consider scaled networks $k\network =
  (\tree,\ambientspaces,kf)$
  for $k \in \N$.
  
  The \textit{variety of tensor network states} $\tns(\network)$
  associated to a tensor network is the set of elements $t \in
  V_\rootvertex \coloneq \bigotimes_{\ell \in \leaves} V_\ell$ such that for each vertex
  $v \in \vertices$, there exists a linear subspace $U_v \subseteq
  V_v \coloneq \bigotimes_{\ell \in \descendants{v}} V_\ell$ with $\dim(U_v) \leq
  f_v$, with $t \in U_\rootvertex$, and so that
  $U_v \subseteq U_{v_1} \otimes U_{v_2}$ if $v$ is not a leaf and 
  $v_1, v_2$ are the children of $v$. 
\end{definition}

The first example of a tensor network variety is the locus of bounded
rank matrices. It arises from the tree $\tree$ on two leaves $\ell, \ell'$
with dimension vector $f = \one$ (i.e., $f_\rootvertex=f_\ell=f_{\ell'}=1$).
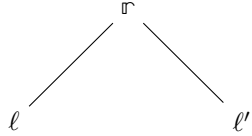
\begin{figure}[H]
\centering
\begin{tikzpicture}[level distance=1.5cm,
  level 1/.style={sibling distance=3cm}]
  \node {$\rootvertex$}
  child {node {$\ell$}}
  child {node {$\ell'$}};
\end{tikzpicture}
\caption{Network for matrices of rank $\le r$.}
\label{fig:matrixrank}
\end{figure}
With this network $\network$, for $r \in \N$ the variety
$\tns(r\network) \subseteq V_\ell \otimes V_{\ell'}$
is the set of
two way tensors of rank $\le r$. In fact, for any dimension vector
$f$, we will obtain this variety: $\tns(\tree,f) = \tns(\tree,r\one)$
for $r=\min\{f_\ell,f_{\ell'}\}$.

\subsection{Containment of Tensor Network State Varieties}
\label{sec:inclusion-setup}
If we want to ask whether one tensor network state variety is
contained in another
\begin{equation*}
  \label{eq:inclusion-question}
  \tns(\network) \ \overset{?}{\subseteq} \tns(\network') \,,
\end{equation*}
the two should live in the same ambient space to begin with.
To this end, we need a bijection $\pi \colon \leaves \to \leaves'$
between the respective leaf sets
together with identifications $V_\ell = V'_{\pi(\ell)}$ for all $\ell
\in \leaves$.
If we are working with plane trees, there is a natural identification
$\pi \colon \leaves \to \leaves'$ going through both sets of leaves
from left to right.

Clearly, $\tns(\tree,\ambientspaces,f) \subseteq
\tns(\tree,\ambientspaces,f')$ provided $f \le f'$  
component wise. The matrix example above shows that the varieties may
be identical even though $f \neq f'$.

Our study was inspired by the results
of~\cite{BBM15} 
which in turn were inspired by Hackbusch's
conjecture (below) about the containment questions for
$\hierarchicalnetwork_k = (\hierarchicaltree_k,\ambientspaces,\one)$
and $\ttnetwork_n =
(\tttree_n,\ambientspaces,\one)$.

\begin{prop}
  \label{prop:HackbuschInverse1}
  Let $k \geq 3$. Consider $\tttree_{2^k}$ and $\hierarchicaltree_k$
  as plane trees with the corresponding leaf identifications.
  Then the inclusion 
  $$\tns(r\,\ttnetwork_{2^k}) \subseteq \tns(r^2\,\hierarchicalnetwork_k)$$
  holds for all $r \in \N$.
\end{prop}

\begin{prop}
  \label{prop:HackbuschInverse2}
  Let $k \in \N$.
  Consider $\tttree_{2^k}$ and $\hierarchicaltree_k$
  as plane trees with the corresponding leaf identifications.
  Then for all $r \in \N$, the inclusion 
  $$ \tns(r\, \hierarchicalnetwork_k) \subseteq \tns(r^{\lceil
    \frac{k}{2} \rceil}\, \ttnetwork_{2^k})$$ 
  holds for all $r \in \N$.
\end{prop}

Moreover, these bounds are tight, which was shown in \cite{BBM15} and
answered a conjecture first posed by Hackbusch~\cite{Hac12}:

\begin{theorem}[\textit{Hackbusch Conjecture}, Theorem 4.3,
  \cite{BBM15}]
  For any $k \geq 1$, $r \geq 2$, the variety
  $\tns(r\hierarchicalnetwork_k)$ is not contained in any of the
  varieties $\tns((r^{\lceil \frac{k}{2} \rceil}-1)\, \ttnetwork_{2^k})$
  for any matching of the leaves.
\end{theorem}

\begin{prop}[Proposition 4.6, \cite{BBM15}]
For any $k \geq 3$, $r \geq 2$, the variety $\tns(r\, \ttnetwork_{2^k})$
is not contained in any of the varieties
$\tns((r^2-1)\, \hierarchicalnetwork_k)$ for any matching of the leaves.
\end{prop}

\subsection{Contractions and DOAD Sets}
The definition of $\tns(\network)$ is an inner description, i.e., a
parameterization as the image of a map. For an outer description with
implicit conditions satisfied by its elements, we need tensor
contractions.

\begin{definition}
The \textit{contraction map of tensors} $\contract$ for vector spaces
$V_1, V_2$ is the map defined on decomposable tensors by
\begin{align*}
\contract : V_1^* \otimes (V_1 \otimes V_2 ) & \to V_2\\
g \otimes (v_1 \otimes v_2) & \mapsto g \contract(v_1 \otimes v_2) := g(v_1)v_2
\end{align*}
and extended linearly. 
For $t \in V_1 \otimes V_2$ we will use the shorthand $V_1^* \contract
t$ for the image of the map
\begin{align*}
V_1^* & \to V_2\\
g & \mapsto g \contract t \,.
\end{align*}
\end{definition}

Let us note right away that
\begin{align*}
V_2^* & \to V_1\\
g & \mapsto g \contract t 
\end{align*}
is the dual/transpose linear map, and in particular has the same rank.

Here comes the promised outer description.

\begin{prop}[Proposition~2.6, Lemma~2.8~\cite{BBM15}] \label{prop: equivalentdef} 
  Let $\network = (\tree,\ambientspaces,f)$ be a tensor network.
  Then the variety of tensor network states $\tns(\network)$ is (set
  theoretically) the locus of tensors $t \in V_\rootvertex$
  such that at any vertex $v \in \vertices$  we have 
  $$
    \dim 
  \Big( V_v^*
  \contract t \Big)
  =
  \dim \biggl( \Big(
  \bigotimes_{\ell \in \antidescendants{v}} V_\ell \Big)^*
  \contract t \biggl) \ \leq \ f_v \,.$$
\end{prop}
In the literature, this dimension is often phrased as the rank of the
flattening of $t$ corresponding to the split
$(\descendants{v}\,|\,\antidescendants{v})$ of $\tree$.

Observe that the condition at the root $\rootvertex$ requires a
subspace of $\kk$ to have dimension at most $f_\rootvertex$. So, as long as
$f$ is positive, the coordinate $f_\rootvertex$ is irrelevant.
The following notion of a doad set will be crucial in our comparison
of two encodings via tensor formats.

\begin{definition}
If we consider a tree $\tree$, then we say a subset of its leaves $S
\subset \leaves$ is a \textit{doad set} (descendant or
anti-descendant leaves) for $\tree$ if there is a vertex $v$ of
$\tree$ such that $S = \descendants{v}$ or $S = \antidescendants{v}$.
We denote the collection of all doad sets by $\doad(\tree)$.
\end{definition}

Observe that the family of descendant sets determines the tree
uniquely while for each of the three different trees on $\leaves =
\{1,2,3\}$ every subset of $\leaves$ is doad.

\begin{lemma} \label{lem: disjointdoads}
  Suppose a proper subset $S \subsetneq \leaves$ is covered by doad
  sets $S_1, \ldots, S_k$ with minimal $k$. Then the $S_i$ are pair
  wise disjoint. 
\end{lemma}

\begin{proof}
Observe that for any $v,w \in \vertices$
\begin{itemize}
\item $\descendants{v}$ and $\descendants{w}$ are either disjoint or one
  is contained in the other, 
\item hence $\antidescendants{v}$ and $\antidescendants{w}$ either
  cover all of $\leaves$ or one is contained in the other, 
\item $\descendants{v} \cup \antidescendants{w} = \leaves$ if $v \ge_\tree
  w$, and
\item $\descendants{v} \subseteq \antidescendants{w}$ if $v$ and $w$
  are incomparable in $\tree$.
\end{itemize}
So, if $S = \bigcup_{i=1}^k S_i$ is a minimal covering by doad sets,
there can be at most one set of the form $\antidescendants{w}$ in the
collection, and then $w$ must be a common ancestor of the remaining
sets.
\end{proof}

We can now generalize the central comparison tool
\cite[Lemma~3.1]{BBM15} to our setting.

\begin{lemma}[{Asymptotic variant of~\cite[Lemma~3.1]{BBM15}}]
  \label{lemma:asymptotic-containment}
Consider two tensor networks $\network = (\tree,\ambientspaces,f)$ and
$\network' = (\tree',\ambientspaces,f')$ on the same set of leaves
$\leaves$.

Now, let $c_0 \in \N$ and assume that for every vertex $v'$ of
$\tree'$ either $\descendants{v'}$ or $\antidescendants{v'}$ is a
union of at most $c_0$ sets $S_i \in \doad(\tree)$. Take $c > c_0$,
then for all $r \gg 0$ 
$$
\tns(r \, \network) \subseteq \tns(r^c \, \network') \,.
$$
\end{lemma}

\begin{proof}
Fix a tensor $t \in \tns(r \, \network)$ and consider any vertex $v'$ of $\tree'$. We need to show that
\[
\dim \biggl( \Big(
  \bigotimes_{\ell \in \descendants{v'}} V_\ell \Big)^*
  \contract t \biggl) \leq r^c f'_{v'} \text{ or, equivalently, }
  \dim \biggl( \Big(
  \bigotimes_{\ell \in \antidescendants{v'}} V_\ell \Big)^*
  \contract t \biggl) \leq r^c f'_{v'}.
\]
By assumption, either $\descendants{v'}$ or $\antidescendants{v'}$ is a union of at most $c_0$ sets $S_i \in \doad(\tree)$. Let $S_i = \descendants{v_i}$ or $S_i = \antidescendants{v_i}$ for some vertex $v_i$ of $\mathcal{T}$. By Lemma \ref{lem: disjointdoads}, we can assume these sets $S_i$ are disjoint. It is then enough to show that
\[
\dim \biggl( 
\bigotimes_{i=1}^{c_0}
\Big(
  \bigotimes_{\ell \in S_i} V_\ell \Big)^*
  \contract t \biggl) \leq r^c f'_{v'}.
\]
By Proposition \ref{prop: equivalentdef}, we have that $\dim 
\Big(  \bigotimes_{\ell \in S_i} V_\ell \Big)^*  \contract t  \leq r f_{v_i}$ for each $i$. Thus, we obtain
\[
\dim \Bigg[ \biggl( 
\bigotimes_{i=1}^{c_0}
\Big(
  \bigotimes_{\ell \in S_i} V_\ell \Big)^* \biggl)
  \contract t \Bigg]
  \le
    \dim \Bigg[ 
\bigotimes_{i=1}^{c_0}
\biggl( \Big(
  \bigotimes_{\ell \in S_i} V_\ell \Big)^*
  \contract t \biggl) \Bigg]
    \leq r^{c_0} f_{v_1} \cdots f_{v_{c_0}} \leq r^c f'_{v'},
\]
where the last inequality follows from the fact that $c_0 < c$ and $r \gg 0$.
\end{proof}

If we want to obtain a non-asymptotic result which contains
Propositions~\ref{prop:HackbuschInverse1}
and~\ref{prop:HackbuschInverse2} as special cases, we need the following notion.

\begin{definition}
Consider two tensor networks $\network = (\tree,\ambientspaces,f)$ and
$\network' = (\tree',\ambientspaces,f')$ on the same set of leaves
$\leaves$.
%
We say that $\tns(\network)$ is \textit{trivially contained} in
$\tns(\network')$ if for every $v' \in \vertices' = \vertices(\tree')$
there is a covering $\descendants{v'} = \bigcup_{i=1}^c S_i$ or a
covering $\antidescendants{v'} = \bigcup_{i=1}^c S_i$ by doad sets of
$\tree$, i.e., there are $v_1, \ldots, v_c \in
\vertices=\vertices(\tree)$ so that $S_i = \descendants{v_i}$ or $S_i
= \antidescendants{v_i}$ for each $i$, with the property
$$
\prod_{i=1}^c f(v_i) \ \le \ f'(v') \,.
$$
\end{definition}

For example, for $f = \one$, $\tns(\network)$ is trivially contained
in any other tensor network variety on $(\leaves, \ambientspaces)$.
($\tns(\tree,\ambientspaces,\one)$ is the locus of decomposable (rank
one) tensors, whatever $\tree$.)

\begin{lemma}
  \label{lemma:trivial}
If $\tns(\network)$ is \textit{trivially contained} in
$\tns(\network')$, then $\tns(\network) \subseteq \tns(\network')$.
\end{lemma}

\begin{lemma}[{Variant of \cite[Lemma~3.1]{BBM15}}] \label{lemma:containment}
Consider two tensor networks $\network = (\tree,\ambientspaces,f)$ and
$\network' = (\tree',\ambientspaces,f')$ on the same set of leaves
$\leaves$, so that $\tns(\network)$ is trivially contained in
$\tns(\network')$.

Now, let $c \in \N$ and assume that for every vertex $v'$ of
$\tree'$ either $\descendants{v'}$ or $\antidescendants{v'}$ is a
union of at most $c$ sets $S_i \in \doad(\tree)$. Then for all $r \in \N$ 
$$
\tns(r \, \network) \subseteq \tns(r^c \, \network') \,.
$$
\end{lemma}

\section{Containment Exponents}

We now introduce one of our main definitions, which will be very
useful when discussing containments of tensor network varieties in
further sections.

\begin{definition}
Let $\tree$ and $\tree'$ be two full binary trees on a common set
$\leaves$ of leaves. A \textit{tensor network exponent} $\tne =
\tne_{\tree,\tree'}$ for $\tree$ in $\tree'$ is a positive number such
that for all $f,f' \in \N^\leaves$, and all $c
> \tne$ there is an $r_0 \in \N$ so that
$$
\tns(r \, \network) \subseteq \tns(r^c \, \network')
$$
for all $\ambientspaces$, and all $r \ge r_0$.
\end{definition}

In this language, the number $c_0$ in
Lemma~\ref{lemma:asymptotic-containment} is a tensor network exponent.
Comparing Lemma~\ref{lemma:asymptotic-containment}
versus Lemma~\ref{lemma:containment}, one can develop a parallel
non-asymptotic theory with non-asymptotic tensor network exponents,
based on trivial containment of the input networks. In the sequel, we
lay out the asymptotic theory only. The non-asymptotic theory does not
need essentially different ideas, but the statements will always make
assumptions on $f,f' \in \N^\leaves$ rather than depending only on
the trees $\tree$ and $\tree'$.

A priory, it is not obvious that for any two trees $\tree, \tree'$
there is a finite exponent. But
Lemma~\ref{lemma:asymptotic-containment} does yield a bound.

\begin{prop} \label{prop:trivial-exponent}
  For two trees $\tree, \tree'$ on a common set of $n$ leaves,
  $\lfloor \frac{n}{2} \rfloor$
  is a tensor network exponent.
\end{prop}

\begin{proof}
  Singletons $\{\ell\} \subseteq \leaves$ are doad. So for every
  vertex $v'$ of $\tree'$, we can cover the smaller of
  $\descendants{v'}$ and $\antidescendants{v'}$ by its one-element
  subsets.
\end{proof}

\subsection{Exponents from the Poset Structure}
We use the language of posets for the descendant relation
$\descendantof$ to formulate a combinatorial bound which is quick to
compute, and which turns out significantly better than
Proposition~\ref{prop:trivial-exponent}.
The tree $\tree$ on the vertex set $\vertices$ yields the structure of
a join-semilattice.
For a set of 
leaves $S \subseteq \leaves$
we denote by $\lca S$ the join, i.e., the lowest common ancestor of $S$.
For $S \subseteq \vertices$, we write $\tree_{\ge S} \coloneq
\{ v \in \vertices \st v \ancestorof s \text{ for some } s \in S \}$
for the set containing all ancestors of all elements of $S$, and
similarly $\tree_{\le S} \coloneq
\{ v \in \vertices \st v \descendantof s \text{ for some } s \in S \}$
for the set containing all descendants of all elements of $S$.
Further, we use $\maxima(\mathcal{P})$ for the number of
maximal elements of the finite poset $\mathcal{P}$.
\begin{theorem} \label{thm:exponents-from-poset}
  Let $\tree$ and $\tree'$ be full binary trees on the common leaf set
  $\leaves$. Then the following number is a containment exponent.
  \[
    \tne_{\tree,\tree'} \coloneq
    \max_{S} 
    \Big\{
    \min \big\{
    \maxima(\vertices \setminus \tree_{\ge S}), \ 
    \maxima(\vertices \setminus \tree_{\ge \leaves \setminus S}), \ 
    \maxima( \tree_{\le \lca(\leaves \setminus S)} \cap \tree_{\ge S}
    ) +1, \ 
    \maxima( \tree_{\le \lca S} \cap \tree_{\ge \leaves \setminus S}
    ) +1
    \big\}
    \Big\} \,,
  \]
  where $S$ runs over all doad sets of $\tree'$.
\end{theorem}

\begin{proof}
We claim that we can cover any non-trivial $S \in \doad(\tree')$ or
its complement $\leaves \setminus S$ with at most
$\tne_{\tree,\tree'}$ doad sets of $\tree$.
Then the assertion follows from Lemma~\ref{lemma:asymptotic-containment}.
Each of the four numbers in the minimum
is the number of doad sets involved in one such covering.

\begin{itemize}
\item $\boldsymbol{ \maxima(\vertices \setminus \tree_{\ge S}):}$
  We cover $\leaves \setminus S$.
  The elements $v$ of $\vertices \setminus \tree_{\ge S}$ are
  precisely those for which $\descendants{v} \subseteq \leaves
  \setminus S$. Then, on the path from any $\ell \in \leaves \setminus
  S$ to $\rootvertex$, we must hit a maximal such element, so that the
  union of these $\descendants{v}$ equals $\leaves \setminus S$.
\item $\boldsymbol{ \maxima( \tree_{\le \lca(\leaves \setminus S)}
    \cap \tree_{\ge S} ) +1:}$
  We cover $S$.
  If $w = \lca(\leaves \setminus S)$, then $\antidescendants{w}
  \subset S$, and the $\descendants{v}$ for $v \in \tree_{\le w} \cap
  \tree_{\ge S}$ cover the rest of $S$. Again, it is enough to take
  $v$ maximal in this subposet.
\item The other two coverings are obtained by taking complements.
\end{itemize}
\end{proof}

Observe that for small $n$, one can record the above minimum for every
non-trivial $S \subseteq \leaves$, and then compare the doad sets of
various $\tree'$ against this data.

\section{Comparison of Plane Tree Networks}
In this section, we focus on plane trees. As mentioned earlier, their
leaves are canonically ordered from left to right, and this yields a
natural way to identify their leaf sets (provided they have the same
number of leaves). In this setting, we can say more about tensor
network exponents than for general trees with general leaf
bijections. One of the big advantages of plane trees is that doad sets
are intervals (or their complements) in the left-to-right ordering of
the leaves.

In a plane tree $\tree$, every vertex $v$ can be identified with a
0/1-sequence, representing the left (0) and right (1) steps on the
path from $\rootvertex$ to $v$ as in Figure~\ref{fig:ex-label}. The
ordering of the leaves is then lexicographic.

\begin{figure}[H]
\centering
  \begin{tikzpicture}[scale=0.5, level distance=2cm,
  level 1/.style={sibling distance=8cm},
  level 2/.style={sibling distance=4cm},
  level 3/.style={sibling distance=2cm}]
  \node {root}
    child {node {$0$}
      child {node {$00$}
      }
      child {node {$01$}
        child {node {$010$}}
        child {node {$011$}}
      }
    }
    child {node {$1$}
    child {node {$10$}
      }
      child {node {$11$}
      child {node {$110$}}
        child {node {$111$}}
      }
    };
\end{tikzpicture}
  \caption{Labelled binary tree $\tree$.}
  \label{fig:ex-label}
\end{figure}
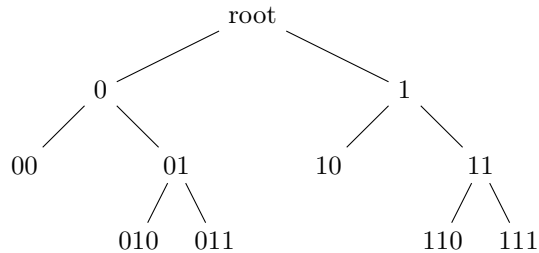

Our strategy to bound tensor network exponents is to compare all plane
trees to the plane train track tree $\tttree_n$
(cf.~\S\ref{sec:trees}).
Then we can use that tensor network exponents behave
transitively in the following sense.

\begin{remark} \label{rk:transitive}
Given trees $\tree, \tree', \tree''$, a tensor
network exponent $\tne_{\tree,\tree'}$ for $\tree$ in $\tree'$, and a
tensor network exponent $\tne_{\tree',\tree''}$ for $\tree'$ in
$\tree''$, the number $\tne_{\tree,\tree''} \coloneq
\tne_{\tree,\tree'} \cdot \tne_{\tree',\tree''}$ is a tensor network
exponent for $\tree$ in $\tree''$.
\end{remark}

\subsection{Containment in TT Networks}
First, we try to include $\tns(\network)$ in
$\tns(\ttnetwork_n)$. Here, we know the doad sets we need to cover.
\begin{remark} \label{rk:tt-doad-sets}
  Identifying the ordered leaf set of the plane train track tree
  $\tttree_n$ with $\{1,\ldots,n\}$, the non-empty doad sets of
  $\tttree_n$ are precisely the sets of the form $\{1,\ldots,\ell\}$
  or $\{\ell,\ldots,n\}$ for $1 \le \ell \le n$.
\end{remark}

To formulate our bound, we need the notion of height of a vertex.

\begin{definition}
  The height $h_v$ of a vertex $v$ of a plane tree $\tree$ is the
  number of 1s occurring before the final 0 in its
  0/1-encoding.\footnote{If the label ends on 0, this is just the
    total number of 1s. If it ends on a couple of consecutive 1s, they
    do not contribute to the height.}
  The dual height $h^*_v$ is the number of 0s occurring before the
  final 1.
\end{definition}

For example, the leaf labeled 01 in Figure~\ref{fig:ex-label} has
height 0 while the leaf 110 has height 2. The sequence of heights for
the leaves of tree in Figure~\ref{fig:ex-label} is 0,1,0,1,2,0.

\begin{theorem} \label{thm:exponent-via-heights}
  Let $\tree$ be a plane tree on $\leaves = \{ 1, \ldots, n \}$. Then
  for $\ell=1, \ldots, n$, the set $\{1,\ldots,\ell\} \subseteq
  \leaves$ can be covered by at most $1 + h_\ell$ descendant sets,
  i.e., sets of the form $\descendants{v}$.
\end{theorem}

\begin{proof}
  We proceed by induction on $n$. For $n=2$, $\tree$ is the tree from
  Figure~\ref{fig:matrixrank}, and every non-empty leaf set is a
  descendants set.

  So suppose $n \ge 3$ and consider leaf $\ell$. We abbreviate $S
  \coloneq \{1,\ldots,\ell\}$.
  If $\ell=1$ or $\ell=n$, the set $S$ itself is a descendant
  set. Otherwise, the 0/1 encoding of our leaf has at least two
  digits.
  We distinguish three cases, based on the last two such digits.

  \begin{itemize}
  \item $\boldsymbol{1:}$ If the last digit is a 1, and $v$ is the
    parent of leaf $\ell$, we have $h_\ell=h_v$.
    The tree $\tree' \coloneq \tree \setminus \tree_{<v}$ obtained by
    removing all proper descendants of $v$ will have strictly fewer
    than $n$ leaves. By induction, we can cover all leaves of $\tree'$
    weakly to the left of leaf $v$ by at most $1+ h_v$ descendant sets
    of $\tree'$.

    If in this covering, $v$ is covered by $\descendants[\tree']{w}$,
    we can take the same covering but replace
    $\descendants[\tree']{w}$ by $\descendants[\tree]{w}$ to obtain
    a covering of $S$.

  \item $\boldsymbol{00:}$ If the last two digits are 00, we are in a
    situation as in Figure~\ref{fig:prooflabel00} where we call $v$
    the parent of leaf $\ell$, $u$ the parent of $v$, $u'$ the right child
    of $u$, and $v'$ the right child of $v$.

    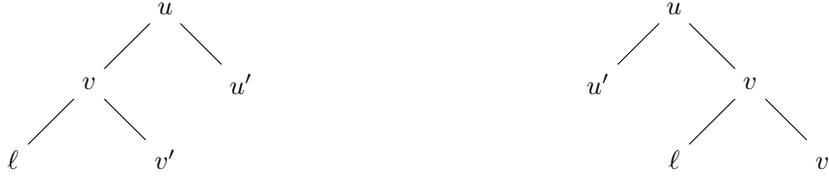
\begin{figure*}[t!]
    \centering
    \begin{subfigure}[t]{0.5\textwidth}
        \centering
      \begin{tikzpicture}[scale=0.5, level distance=2cm,
        level 1/.style={sibling distance=4cm},
        level 2/.style={sibling distance=4cm},
        level 3/.style={sibling distance=2cm}]
        \node {$u$}
        child {node {$v$}
          child {node {$\ell$}}
          child {node {$v'$}}
        }
        child {node {$u'$}
        };
      \end{tikzpicture}
    \caption{Structure near leaf $\ell$ when its labelling ends in $00$.}
      \label{fig:prooflabel00}
    \end{subfigure}%
    ~ 
    \begin{subfigure}[t]{0.5\textwidth}
        \centering
      \begin{tikzpicture}[scale=0.5, level distance=2cm,
        level 1/.style={sibling distance=4cm},
        level 2/.style={sibling distance=4cm},
        level 3/.style={sibling distance=2cm}]
        \node {$u$}
        child {node {$u'$}
        }
        child {node {$v$}
          child {node {$\ell$}
          }
          child {node {$v'$}
          }
        };
      \end{tikzpicture}
      \caption{Structure near leaf $\ell$ when its labelling ends in $10$.}
      \label{fig:prooflabel10}
      \end{subfigure}
      \caption{Two cases for the leaf $\ell$.}
\end{figure*}
    
    We have $h_\ell=h_v=h_u$.
    The tree $\tree' \coloneq \tree \setminus \tree_{<v}$ obtained by
    removing all proper descendants of $v$ will have strictly fewer
    than $n$ leaves. By induction, we can cover all leaves of $\tree'$
    weakly to the left of leaf $v$ by at most $1+ h_v$ descendant sets
    of $\tree'$. Now the only descendant set containing $v$ but not
    containing $u'$ is $\descendants[\tree']{v} = \{v\}$. If we
    replace $\descendants[\tree']{v}$ by $\descendants[\tree]{\ell} =
    \{\ell\}$, we obtain a covering of $S$.
  \item $\boldsymbol{10:}$  If the last two digits are 10, we are in a
    situation as in Figure~\ref{fig:prooflabel10} where we call $v$
    the parent of leaf $\ell$, $u$ the parent of $v$, $u'$ the left child
    of $u$, and $v'$ the right child of $v$.
    We have $h_\ell=h_{u'}+1$. The label of leaf $\ell-1$ is the label of
    $u'$ possibly with a sequence of 1s appended. Hence,
    $h_{\ell-1}=h_{u'}$.
    The tree $\tree' \coloneq \tree \setminus \tree_{<v}$ obtained by
    removing all proper descendants of $v$ will have strictly fewer
    than $n$ leaves. By induction, we can cover all leaves of $\tree'$
    weakly to the left of leaf $\ell-1$ by at most $1+ h_{\ell-1}$
    descendant sets of $\tree'$.
    Add to this covering the set $\descendants[\tree]{\ell}$ 
    to obtain a covering of $S$ by $2+h_{\ell-1}=1+h_\ell$ descendant sets.
  \end{itemize}
\end{proof}

The previous theorem does not take advantage of the flexibility we
have if we want to apply Lemma~\ref{lemma:asymptotic-containment} in
our situation. For one, we can cover either $\{1,\ldots,\ell\}$ or
$\{\ell+1,\ldots,n\}$. For another, we can use general doad sets in the
covering. The first asymmetry is easily addressed.

\begin{corollary} \label{cor:exponent-via-heights-symmetric}
  Let $\tree$ be a plane tree on $\leaves = \{ 1, \ldots, n \}$. Then
  \begin{equation} \label{eq:exponent-via-heights-symmetric}
    1 + \max \left\{ \min \{h_\ell,h^*_{\ell+1}\} \st \ell \in \leaves
      \setminus \{n\}
      \right\}
  \end{equation}
  is a tensor network exponent for $\tree$ in $\tttree_n$.
\end{corollary}

This already yields pretty decent bounds in examples.

\begin{proof}
  We can flip $\tree$, exchanging left and right children everywhere
  to the dual tree $\tree^*$.
\end{proof}

In particular, we can recover Proposition~\ref{prop:HackbuschInverse1}
using Corollary~\ref{cor:exponent-via-heights-symmetric}.

\begin{proof}[Proof of Proposition~\ref{prop:HackbuschInverse1}]
  In $\hierarchicaltree_k$, the binary label of leaf $\ell \in \{0,
  \ldots, 2^k-1\}$ is just the encoding of the number $\ell$ as a
  $k$-digit binary number. If $h_\ell > 0$ it will have the form
  $\omega 0 1 \ldots 1$ for some 0/1 string $\omega$ followed by the
  final 0 followed by a non-negative number of 1s. Then the label
  of leaf $\ell+1$ is $\omega 1 0 \ldots 0$. So $h_\ell$ equals the
  number of 1s in $\omega$, and $h^*_{\ell+1}$ equals the number 0s in
  $\omega$. We get $h_\ell + h^*_{\ell+1} \le k-1$ so that $\min
  \{h_\ell, h^*_{\ell+1}\} \le \frac{k-1}{2}$.
\end{proof}

\subsection{Containment of TT Networks}
It is actually easy to embed TT network varieties into varieties of
other networks. The following result generalizes
Proposition~\ref{prop:HackbuschInverse1}.

\begin{theorem} \label{thm:TT-exponent-2}
  Let $\tree$ be a plane tree on $\leaves = \{ 1, \ldots, n \}$.
  Then $2$ is a tensor network exponent for $\tttree_n$ in $\tree$.
\end{theorem}

Using Remark~\ref{rk:transitive}, we obtain a general comparison bound
for plane trees.

\begin{corollary} \label{cor:exponent-plane-general}
  Let $\tree, \tree'$ be plane trees on accordingly ordered $\leaves =
  \{ 1, \ldots, n \}$. Then
  \begin{equation} \label{eq:exponent-via-heights-symmetric}
    2 + 2 \max \left\{ \min \{h_\ell,h^*_{\ell+1}\} \st \ell \in
      \leaves \right\}
  \end{equation}
  is a tensor network exponent for $\tree$ in $\tree'$.
\end{corollary}

\begin{proof}[Proof of Theorem~\ref{thm:TT-exponent-2}]
  As mentioned before, all doad sets of the plane tree $\tree$ are
  contiguous intervals in $\{ 1, \ldots, n \}$. If such an interval
  contains $1$ or $n$, it is doad for $\tttree_n$. Otherwise, its
  complement is a union of such intervals in $\doad(\tttree_n)$.
\end{proof}

\begin{remark}
The bound obtained from Corollary \ref{cor:exponent-plane-general} is not sharp in general. This is trivially true in the special case where $\tree' = \tttree_n$ for $n$ the number of leaves of $\tree$ and $\tree'$. This bound can differ greatly from the smallest containment exponent: consider, for instance, the trees presented in Figure \ref{fig:bound-not-sharp}. Using Lemma \ref{lemma:asymptotic-containment}, or the bound obtained from the poset structure, we find that the exponent for the containment of the tree $\tree$ in Figure \ref{subfig:a} in $\tree'$, the tree in Figure \ref{subfig:b}, is already $1$ while Corollary \ref{cor:exponent-plane-general} provides $4$ as the upper bound. 
\end{remark}

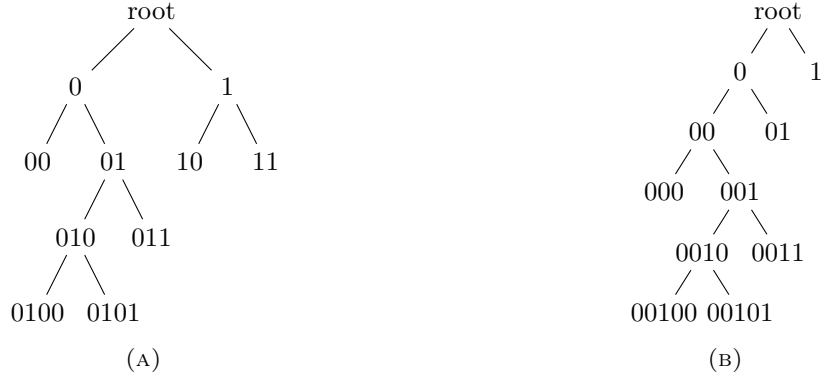
\begin{figure}[t!] 
    \centering
    \begin{subfigure}[t]{0.5\textwidth}
        \centering
\begin{tikzpicture}[scale=0.5, level distance=2cm,
  level 1/.style={sibling distance=4cm},
  level 2/.style={sibling distance=2cm},
  level 3/.style={sibling distance=2cm}]
  \node {root}
    child {node {$0$}
      child {node {$00$}
      }
      child {node {$01$}
        child {node {$010$}
        child {node {$0100$}}
        child {node {$0101$}
        }
        }
        child {node {$011$}}
      }
    }
    child {node {$1$}
    child {node {$10$}}
      child {node {$11$}
      }
    };
\end{tikzpicture}
        \caption{}
        \label{subfig:a}
    \end{subfigure}%
    ~ 
    \begin{subfigure}[t]{0.5\textwidth} 
        \centering
\begin{tikzpicture}[scale=0.5, level distance=1.6cm,
  level 1/.style={sibling distance=2cm},
  level 2/.style={sibling distance=2cm},
  level 3/.style={sibling distance=2cm}]
  \node {root}
    child {node {$0$}
      child {node {$00$}
      child {node {$000$}}
        child {node {$001$}
        child {node {$0010$}
        child {node {$00100$}}
        child {node {$00101$}}
        }
        child {node {$0011$}}
        }
      }
      child {node {$01$}
      }
    }
    child {node {$1$}
    };
\end{tikzpicture}        
        \caption{}
         \label{subfig:b}
    \end{subfigure}
    \caption{Two trees with 6 leaves.}
    \label{fig:bound-not-sharp}
\end{figure}


\section{Computational Exponent Search}
In the previous sections, we have developed the purely
combinatorial way to certify tensor network containments.
To develop an intuition about this new measure, we determine
the bound of Lemma~\ref{lemma:asymptotic-containment} for all pairs of
trees on small leaf sets.

\subsection{Trees on few leaves}


For our implementation we have worked with the poset structure of the full binary trees using the Sage package \textit{BinaryTree} whose documentation can be found in \cite{SageBin}. One can either work with the binary tree structure, or convert it to a poset. Hence, the key feature of this implementation is that the algorithm generates all full binary trees on $n$ leaves for a given $n$.\\

The number $\tau_n$ of such trees on $n \ge 2$ leaves 
$$ 1, 1, 2, 3, 6, 11, 23, 46, 98, 207, 451, 983, \ldots $$
is sequence~\cite[A001190]{OEIS}.
In order to cover all possible tensor network comparisons
$(\tree,\tree')$, it is enough to permute the leaf set of $\tree'$.
We thus get $\tau_n \times \tau_n \times n!$ instances of the following
problem ($\doad(\tree, \tree', \pi)$) for trees $\tree$, $\tree'$ on
$\leaves = \{1, \dots, n\}$ and a permutation $\pi \colon \leaves \to
\leaves$.
\begin{quote}
  Find for every subset $S \subset \leaves$ the minimal number $n_S$
  of doad sets of $\tree$ whose union is $S$. Then record $\max \{ n_S
  \st \pi(S) \in \doad(\tree') \}$.
\end{quote}

For instance, for $n=4$, we obtain the following $2 \times 2 \times
24$ tensor of tensor network exponents which can be described by the matrix
\[
\begin{bmatrix}
    1 & 1 \\
    1 & 1 \\
\end{bmatrix}
\]
once per each of the 24 permutations.

The results for $n=5, 6, 7, 8$ are collected in \cite{sofiaGitHub}.\\



We end this section with a remark.

\begin{remark} 
The bound from Theorem~\ref{thm:exponents-from-poset} is not necessarily sharp. Consider the binary tree $\mathcal{T}$ as in Figure \ref{fig:example8leaves} which we compare to the binary tree $\tttree_8$.

\begin{figure}[H]
\centering
  \begin{tikzpicture}[scale=0.5, level distance=2cm,
  level 1/.style={sibling distance=10cm},
  level 2/.style={sibling distance=4cm},
  level 3/.style={sibling distance=3cm},
  level 4/.style={sibling distance=3cm},
  level 5/.style={sibling distance=3cm},
  level 6/.style={sibling distance=2cm},]
  \node {Root}
    child {node {$0$}
    child {node {$00$}
      }
      child {node {$01$}
      child {node {$010$}
      }
      child {node {$011$}
      child {node {$0110$}
      child {node {$01100$}
      }
      child {node {$01101$}
      }
      }
      child {node {$0111$}
      }
      }
      }
    }
    child {node {$1$}
    child {node {$10$}
      child {node {$100$}
      }
      child {node {$101$}
      }
      }
      child {node {$11$}
      }
    };
\end{tikzpicture}
  \caption{A binary tree on $8$ leaves.}
  \label{fig:example8leaves}
\end{figure}
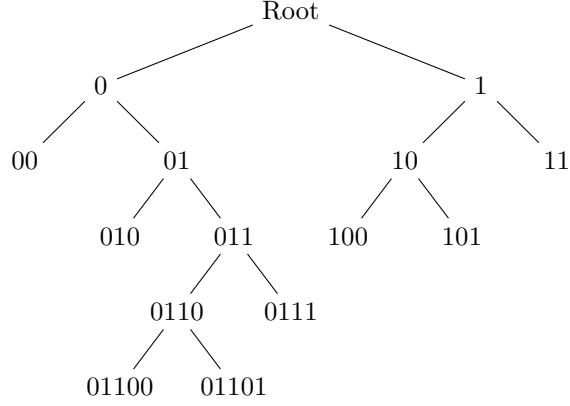

Here, using Theorem~\ref{thm:exponents-from-poset}, the leaf $l = 01101$ will give us the bound $\tne_{\tree,\tree'} = 3$. However, a generic tensor $t$ in $ \tns(\network)$ satisfies
\[
\dim \biggl( \Big(
  \bigotimes_{\ell \in \antidescendants{v}} V_\ell \Big)^*
  \contract t \biggl) \ \leq \ f_v \,,
\]
where $f_v = 2$ for all $v \in \mathcal{V}$.
\end{remark}

\subsection{Containment exponents via integer programming}
For larger $n$, it becomes infeasible to compare all doad sets of
$\tree$ and $\tree'$ against each other. We now describe an integer
linear program which computes this bound.
We introduce decision variables $\underline{x}^w_v,
\overline{x}^w_v, \underline{y}^w_v, \overline{y}^w_v,
\underline{z}^w, \overline{z}^w \in \{0,1\}$
for vertices $v$ of $\tree$ and nodes $w$ of $\tree'$ with the
interpretation:
\\[.5ex]
\begin{minipage}[t]{.6\textwidth}
\vspace{-0.2cm}
\begin{align*}
    \underline{z}^w &= 1 \text{ if we cover $\descendants[\tree']{w}$},\\
    \overline{z}^w & =1 \text{ if we cover $\antidescendants[\tree']{w}$},\\
    \underline{x}^w_v & =1 \text{ if we use $\descendants[\tree]{v}$ in the
  cover of $\descendants[\tree']{w}$},\\
  \overline{x}^w_v & =1 \text{ if we use $\descendants[\tree]{v}$ in the
  cover of $\antidescendants[\tree']{w}$},\\
  \underline{y}^w_v & =1 \text{ if we use $\antidescendants[\tree]{v}$ in the
  cover of $\descendants[\tree']{w}$},\\
  \overline{y}^w_v & =1 \text{ if we use $\antidescendants[\tree]{v}$ in the
  cover of $\antidescendants[\tree']{w}$}.\\
\end{align*}
\end{minipage}
\begin{minipage}[t]{.3\linewidth}
\vspace{0.3cm}
  \centering
  or, more compactly, 
  \smallskip
  
  \begin{tabular}{ c|c|c|c } 
    & $\descendants[\tree]{v}$ & $\antidescendants[\tree]{v}$ \\ 
    \hline
    $\descendants[\tree']{w}$ & $\underline{x}^w_v$ & $\underline{y}^w_v$ & $\underline{z}^w$ \\ 
    \hline
    $\antidescendants[\tree']{w}$ & $\overline{x}^w_v$ & $\overline{y}^w_v$ & $\overline{z}^w$
  \end{tabular}
\end{minipage}
\medskip

Some of these variables are not needed, and can be set to zero, namely,
\begin{equation*}
\underline{x}^w_v = 0 \text{ if } \descendants[\tree]{v} \nsubseteq
\descendants[\tree']{w}, \ 
\overline{x}^w_v = 0 \text{ if } \descendants[\tree]{v} \nsubseteq
\antidescendants[\tree']{w}, \ 
\underline{y}^w_v = 0 \text{ if } \antidescendants[\tree]{v} \nsubseteq
\descendants[\tree']{w}, \ 
\overline{y}^w_v = 0 \text{ if } \antidescendants[\tree]{v} \nsubseteq
\antidescendants[\tree']{w}.
\end{equation*}

We need one more objective variable $c$, bounding the number of doad
sets needed in the various covers. Our goal is then to minimize $c$,
subject to the constraints
\begin{equation} \label{eq:IP} \tag{IP($\tree,\tree'$)}
  \begin{array}{rcl@{\quad}l} 
    \underline{z}^w + \overline{z}^w
    & \ge & 1 & \text{ for all } w \in \nodes(\tree')
    \\[1ex]
    \displaystyle \sum_{\descendants[\tree]{v} \ni \ell} \underline{x}^w_v +
    \sum_{\antidescendants[\tree]{v} \ni \ell}
    \underline{y}^w_v
    & \geq & \underline{z}^w
    & \text{ for all pairs } \ell \in \leaves, w \in \nodes(\tree')
      \text{ with } \ell \in \descendants[\tree']{w}
    \\
    \displaystyle \sum_{\descendants[\tree]{v} \ni \ell} \overline{x}^w_v +
    \sum_{\antidescendants[\tree]{v} \ni \ell} \overline{y}^w_v
    & \geq & \overline{z}^w
    & \text{ for all pairs } \ell \in \leaves, w \in \nodes(\tree')
      \text{ with } \ell \in \antidescendants[\tree']{w}
    \\[.5ex]
    \displaystyle \sum_{v} ( \underline{x}^w_v + \underline{y}^w_v )
    & \leq & c
    & \text{ for all } w \in \nodes(\tree') \\
    \displaystyle \sum_{v} ( \overline{x}^w_v +  \overline{y}^w_v )
    & \leq & c
    & \text{ for all } w \in \nodes(\tree')
  \end{array}
\end{equation}

Any feasible solution to this integer program yields a tensor network
exponent.

\section{Concluding Remarks}

The collection $\ambientspaces$ of vector spaces did not really play a
role in our considerations. In fact, for fixed $\ambientspaces$,
$\tns(r\network)$ will simply fill the ambient space for high enough
$r$.

Considering finite dimensional $V_\ell$ will introduce interesting
boundary effects. But as long as $f_v \le \dim V_v$%
, these will not play a role.

The minimal tensor network exponent for $\tree$ in $\tree'$ yields a
new (non-symmetric) distance measure for trees on a fixed leaf set
$\leaves$. It will be interesting to compare it to other tree
distances in the literature.

\subsection{Open Questions}

We end this paper with some open questions.

\begin{question}
Is it true that the upper bound for tensor network exponents from
Lemma~\ref{lemma:asymptotic-containment} is always tight as it is for
$\tttree_n$ vs.\ $\hierarchicaltree_k$?  
\end{question}

Additionally, any feasible solution to the dual of the LP relaxation
of~\eqref{eq:IP} yields a lower bound on the number of doad sets
needed for a cover. 

\begin{question}
Is it possible to use such a dual solution in order to construct 
tensors which yield lower bounds to (asymptotic) tensor network exponents?
\end{question}



\bibliographystyle{amsalpha}
\bibliography{biblio}

@article{BBM15,
  title={The {Hackbusch} conjecture on tensor formats},
  author={Buczy{\'n}ska, Weronika and Buczy{\'n}ski, Jaros{\l}aw and Micha{\l}ek, Mateusz},
  journal={Journal de Math{\'e}matiques Pures et Appliqu{\'e}es},
  volume={104},
  number={4},
  pages={749--761},
  year={2015},
  publisher={Elsevier}
}

@book{Hac12,
  title={Tensor spaces and numerical tensor calculus},
  author={Hackbusch, Wolfgang},
  volume={42},
  year={2012},
  publisher={Springer}
}

@misc{SageBin,
  author =	 {Tom Boothby},
  title =	 {Binarytrees {Sage} Package Documentation, Version 10.6 Reference Manual},
  howpublished = {\url{https://doc.sagemath.org/html/en/reference/data_structures/sage/misc/binary_tree.html}},
  year = 2007,
  note =	 {Accessed: 2025-08-26}
}

@misc{sofiaGitHub,
  author       = {Garz\'on Mora, Sof\'ia},
  title        = {Tensor Network Varieties},
  howpublished = {\url{https://github.com/sofiagarzonmora/Tensor-Network-Varieties}},
  year         = {2025},
  note         = {GitHub repository; accessed 26 August 2025},
}

@Misc{OEIS,
  author =	 {OEIS Foundation Inc.},
  title =	 {The On-Line Encyclopedia of Integer Sequences},
  howpublished = {\url{https://oeis.org/}},
  year =	 2025
}

\end{document}